\documentclass[times, 10pt,twocolumn]{article}
\usepackage{latex8}
\usepackage{times}

\usepackage{amsfonts,amsmath,amssymb,algorithm,epsfig,fancyhdr,graphicx,mathrsfs,eucal}
\usepackage{epsf}
\usepackage{url}

\newcommand{\RR}{\ensuremath{\mathbb{R}}}

\newcommand{\ZZ}{\ensuremath{\mathbb{Z}}}

\newcommand\inter{\operatorname{int}}

\newcommand\conv{\operatorname{\mathsf{conv}}}
\newcommand\aff{\operatorname{\mathsf{aff}}}

\newcommand\rank{\operatorname{rank}}
\newcommand\qrank{\operatorname{\mathsf{qrank}}}

\newcommand\Quad{\operatorname{Quad}}
\newcommand\Lin{\operatorname{Lin}}
\newcommand\armin{\operatorname{armin}}
\newcommand\trace{\operatorname{trace}}
\newcommand\Sym{\operatorname{\textsf{Sym}}}
\newcommand\Const{\operatorname{Const}}

\newcommand\vertex{\operatorname{\mathsf{vert}}}

\renewcommand{\phi}{\varphi}                 
\renewcommand{\epsilon}{\varepsilon}

\newcommand{\cc}{\mbox{\boldmath$c$}}

\newcommand{\e}{\mbox{\boldmath$e$}}

\newcommand{\p}{\mbox{\boldmath$p$}}

\newcommand{\bt}{\mbox{\boldmath$t$}}

\newcommand{\vv}{\mbox{\boldmath$v$}}
\newcommand{\uu}{\mbox{\boldmath$u$}}

\newcommand{\x}{\mbox{\boldmath$x$}}

\newcommand{\z}{\mbox{\boldmath$z$}}

\newcommand\bA{\mathbf{A}}

\newcommand\bB{\mathbf{B}}

\newcommand\E{\operatorname{E}}

\newcommand\R{\mathbb{R}}

\newcommand\Z{\mathbb{Z}}

\newcommand\cB{\mathcal{B}}

\newcommand\cD{\mathcal{D}}
\newcommand\cF{\textsf{Q}}
\newcommand\cE{\mathcal{E}}
\newcommand\cH{\mathcal{H}}

\newcommand\cQ{\mathcal{Q}}

\newcommand\cV{\mathcal{V}}

\newtheorem{theorem}{Theorem}

\newtheorem{conjecture}{Conjecture}

\begin{document}

\title{A New Algorithm in Geometry of Numbers}

\author{
Mathieu Dutour\\
Laboratory of Radiochemistry\\
Rudjer Boskovic Institute\\
Bijenicka 54, 10 000 Zagreb, Croatia\\
\url{Mathieu.Dutour@ens.fr}\\
\and
Konstantin Rybnikov\\
University of Massachusetts at Lowell\\
Mathematical Sciences\\ One University Ave., MA 01854, USA\\
\url{Konstantin_Rybnikov@uml.edu}\\
}

\maketitle \thispagestyle{empty}


\urlstyle{sf}

\begin{abstract}
A lattice Delaunay polytope $P$ is called \emph{perfect} if its Delaunay sphere is the only ellipsoid circumscribed about $P$. We present a new
algorithm for finding perfect Delaunay polytopes. Our method overcomes the major shortcomings of the previously used method \cite{DutourAdj}. We have
implemented and used our algorithm for finding perfect Delaunay polytopes in dimensions 6, 7, 8. Our findings lead to a new conjecture that sheds
light on the structure of lattice Delaunay tilings.
\end{abstract}

\section{Introduction}
Let $\Lambda$ be an $n$-dimensional lattice ($n \ge 0$) and let $P \subset \Lambda
\otimes \R \cong \R^n$ be a polytope whose vertex set $\vertex P$ belongs to $\Lambda$.
We say that $P$ is a \emph{Delaunay polytope} for $\Lambda$ if $P$ can be circumscribed
by a closed ball $\cB_P \subset \Lambda \otimes \R$ such that $\cB_P \cap
\Lambda=\vertex P$. The ball $\cB_P$ (or its boundary) is commonly referred to as the
\emph{Delaunay sphere} (or empty sphere) for $P$. (Delaunay \cite{Del} himself
attributed the concept of empty sphere to Voronoi.) Delaunay polytopes for $\Lambda$
form a face-to-face tiling of $\Lambda \otimes \R$ called the \emph{Delaunay tiling
}for $\Lambda$.

One can study the geometry of lattices by comparing their Delaunay tilings. Such study
was initiated by Voronoi \cite{Vor2}. As $\Lambda \subset \Lambda \otimes \R \cong
\R^n$ is deformed into $\Z^n \subset \R^n$ by an affine transformation $\x \mapsto
A(\x)$, the empty spheres circumscribed about the Delaunay polytopes of $\Lambda$ are
deformed into empty ellipsoids circumscribed about the $A$-images of these polytopes.
All these ellipsoids have identical quadratic parts -- indeed they are balls in the
metric $d(\x,\x')=||A^{-1}(\x)-A^{-1}(\x')||$. Thus, the study of Delaunay tilings for
$n$-lattices is equivalent to the study of Delaunay tilings for $\Z^n$ with respect to
different positive definite quadratic forms. Let us denote the Delaunay tiling for
$\Z^n$ with respect to a positive quadratic form $\cQ$ by $Del(\Z^n,\cQ)$. The Delaunay
property of an ellipsoid $\cE(\cQ,\cc,R)=\{\x \in \R^n\;|\; \cQ[\x-\cc] \le R^2\}$,
circumscribed about a polytope $P$, means that the quadratic function $\cQ[\x-\cc]-R^2$
is zero on $\vertex P$ and strictly positive on $\Z^n \setminus \vertex P$. From now on
we will be  working with $\Lambda=\Z^n$, unless stated otherwise.

 It is natural to extend the notions of Delaunay polytope and
tiling to positive semidefinite forms. We refer to an $\R$-valued function $f$ on a set
$S$ as \emph{positive} if $f(x)\ge 0$ for any $x \in S$. Let $\cQ$ be a quadratic form
that is positive on $\R^n$ and such that the rank of the sublattice $\ker_{\R} \cQ \cap
\Z^n$ is equal to the dimension  of $\ker_{\R} \cQ$.  Then $\R^n$ is tiled by unbounded
$n$-dimensional \emph{polyhedra,} which are Delaunay with respect to $\cQ$: each
polyhedron $P$ from this family is circumscribed by an elliptic cylinder $\cE_P$, whose
interior is free of lattice points, so that $P\cap \Z^n =\cE_P\cap \Z^n$. Furthermore,
$P$ is the direct affine product of a Delaunay \emph{polytope }$P$ for a sublattice
$\Lambda \subset \Z^n$ of rank $r=\dim \ker_{\R}\cQ$ and an affine $(n-r)$-subspace of
$\R^n$. The degenerate Delaunay ``ellipsoid" $\cE_P$ for an unbounded polyhedron $P$ is
the direct affine product of the Delaunay ellipsoid for $P$ and $L$; we will be using
`ellipsoid' for both bounded and unbounded ellipsoids. For example, $\R^n$ is tiled by
the \emph{unit slabs} $U_i=\{\x \;\vline\; i \le x_1\le i+1\}$ where $i \in \Z$. Each
unit slab is a Delaunay polyhedron with respect to quadratic form $x_1^2$; furthermore,
each $U_i$ coincides with its Delaunay ellipsoid $\cE_{U_{i}}$. Following a common
convention we will be using the word `polytope' only for bounded polyhedra.

Let $\Lambda$ be a lattice of rank $n$ and let $P \subset \Lambda \otimes \R$ be a
\emph{lattice polyhedron}, i.e., the convex hull of a subset of $\Lambda$. Then $P$ is
called \emph{perfect} if there is an $n$-ellipsoid (possibly degenerate) circumscribed
about $P$ and this ellipsoid is unique. Perfect Delaunay polyhedra are also called
extreme (e.g. \cite{DutourAdj}). Perfect Delaunay \emph{polytopes} are rare in small
dimensions, e.g., for $n \le 6$ there are only three such polytopes -- $0$ for $n=0$,
$[0,1]$ for $n=1$, and Gosset's semiregular polytope $2_{21}$ for $n=6$.
 The previous method
\cite{DutourAdj} for finding perfect Delaunay polytopes was based on
an unproven conjecture that every perfect Delaunay polytope is
\emph{basic.} A lattice polytope $P$ is called \emph{basic} if there
exist $v_0,\dots, v_n \in \vertex P$ such that every $v \in \vertex
P$ can be written as $v=\sum_{i=0}^n \lambda_i v_i$ where
$\lambda_1,\dots,\lambda_{n+1} \in \Z$  and $1=\sum_{i=0}^n
\lambda_i$.
 We know that there exist non-basic Delaunay polytopes in higher dimensions (see
 \cite{grishukhin}) and we cannot rule out the existence of
 non-basic perfect Delaunay polytopes.

The perfection property of a Delaunay polytope $P$ with ellipsoid
$\cE(\cQ,\cc,\rho)=\{\x \in \R^n\;|\; \cQ[\x-\cc] \le \rho^2\}$ amounts to that any
quadratic function that vanishes on $\vertex P$ is of the form  $\alpha(\cQ[\x-\cc] -
\rho^2)$ where $\alpha \in \R$. A real-valued quadratic function $F$ on $\R^n$ is
called \emph{perfect} if $\armin F =\min \{F(\z)\;\vline\;\z \in \Z^n\} \ge 0$ and
$\conv \{\z \in \Z^n\;\vline \;F(\z)=\armin F\}$ is a perfect Delaunay polyhedron. The
ellipsoid $\{\x \in \R^n\;\vline \;F(\x)\le \armin F\}$ is also called perfect.

 Erdahl proved that the vertex set of any perfect Delaunay
\emph{polyhedron} splits uniquely (up to arithmetic equivalence) into the direct affine
sum of the vertex set of a perfect Delaunay polytope and a sublattice which is parallel
to the kernel of $\cQ$ \cite{erdahlcone}.
\begin{theorem}{\protect \emph{\cite{erdahlcone}}} A polyhedron
$P \in Del(\Z^n,\cQ)$ is perfect
 if and only if   \[P \cap \Z^n=\{v + \z \; \vline \; v \in \vertex D,\: \z \in \Gamma\},\]
 where
$D$ is a perfect \emph{polytope } from $Del(\Z^n \cap \aff D,\cQ)$ and $\Gamma$ is a
submodule of $\Z^n$ such that $\Z^n$ is the direct sum of modules $(\Z^n \cap \aff
D)-(\Z^n \cap \aff D)$ and  $\Gamma$. If $(D',\Gamma')$ is another pair with these
properties, then $\Gamma'=\Gamma$ and $D'=A(D)$, where $A$ is an \emph{affine}
automorphism of $\Z^n$.
\end{theorem}

The fundamental importance of perfect Delaunay polyhedra is
explained by the following theorem of Erdahl \cite{Toronto}.
\begin{theorem}
Let $D$ be a Delaunay polytope for $\Z^n$. Then
\[D=\bigcap_{i=1}^k
P_i,\] where each $P_i$ is a perfect Delaunay polyhedron for $\Z^n$
with respect to some positive form $\cQ_i$ such that $\dim
\ker_{\R}\cQ_i=\rank (\ker_{\R}\cQ_i \cap \Z^n)$.
\end{theorem}

 For $S \subset \Z^n$ we define $\qrank S$, the \emph{quadratic}
 \emph{rank} of $S$, as the dimension of the space of quadratic functions on
 $\R^n$  that vanish on $S$.
  Let $P$ and $P_1$ be two perfect Delaunay
\emph{polytopes} for $\Z^n$ such that $\qrank (\vertex P \cap \vertex P_1)=2$. In this
case we call $P$ and $P_1$ \emph{adjacent.} If two perfect Delaunay $n$-polytopes $P$
and $P'$ can be connected by a sequence of perfect Delaunay $n$-polytopes in which
every two consecutive members are adjacent, then we will say that $P$ and $P'$ belong
to the same \emph{adjacency component}. We developed a method for finding an adjacency
component for a perfect Delaunay polytope. We have found that for each $n \le 8$ all
known perfect Delaunay $n$-polytopes belong to the same adjacency component. This
finding makes compelling the following conjecture.

\begin{conjecture} For any $n \in \mathbb{N}$ all perfect Delaunay $n$-polytopes
belong to the same adjacency component.
\end{conjecture}

\section{Space of quadratic functions}
Let us denote by $\Sym(n)$ the space of real symmetric $n \times n$ matrices; by
interpreting an element of $\Sym(n)$ as a Gram matrix, we can  regard $\Sym(n)$ as the
space of quadratic forms with real coefficients. Denote by $\cF(n)$ the linear space of
quadratic functions on $\R^n$ and by $\cF_0(n) \subset \cF(n)$ the subspace of
functions with zero constant term. Since a quadratic function can be represented
uniquely as the sum of a quadratic form, a linear functional, and a constant, it is
convenient to introduce the projection operators
\[\Quad:\cF(n) \rightarrow \Sym(n),\; \Lin:\cF(n) \rightarrow {\R^n}^*,\;\]
\[\Const:\cF(n) \rightarrow \R.\] For two quadratic forms with Gram matrices $\bA$ and
$\bB$ we define the dot product on $\Sym(n)$ as $\trace(\bA\bB)$. For linear functions
defined by covectors $a$ and $b$ the dot product is just $a\cdot b$. The dot product on
$\cF_0(n)$ is defined as the direct sum of the dot products on $\Sym(n)$ and
${\R^n}^*$.

The main idea of this paper is in interpreting quadratic functions on $\R^n$ as
elements of $\cF_0(n)^*$, the dual of $\cF_0(n)$. There is a natural correspondence
between ellipsoids in $\R^n$ and closed (affine) halfspaces of $\cF_0(n)$. Namely, if
$\cE=\{\x \in \R^n\;|\; \cQ[\x-\cc] \le \rho^2\}$, then the corresponding halfspace
$H_{\cE}$ is $\{X \in \cF_0(n)\; | \; X \cdot F \ge \rho^2 - \cQ[\cc]\}$, where $F$ is
the quadratic function defined by $F(\x)=\cQ[\x]-2\cQ(\x,\cc)$.

Let $\cD$ be a map from $\R^n$ into
 $\cF_0(n)$ defined in the matrix notation by
\[\cD: \uu \mapsto (\x \mapsto  \x^T(\uu\uu^T)\x+\uu^T\x),\]
where $\x$ and $\uu$ are treated as column vectors. Obviously, $\cD$ takes an integer
vector to a quadratic function with integer Gram matrix and integer linear part; such
quadratic function is called classically integer. The map $\cD$ resembles the Voronoi
map $\cV:\R^n \rightarrow \Sym(n)$ that takes a vector $\uu$ to the quadratic form with
Gram matrix $\uu\uu^T$ (see \cite{RB79} for details). Thus, we have $\cD(\uu)= \cV(\uu)
+\uu^*$, where $\uu^*$ is linear functional dual to $\uu$. The map $\cV$ can be seen as
the quadratic Veronese map from $\R^n$ to  $\Sym(n)$, although in contemporary
literature the Veronese map is usually defined in the projective setup.
We call an ellipsoid in $\R^n$ \emph{empty} if its interior is free of points of
$\Z^n$.
 If $\cE=\{\x \in \R^n\;|\;\cQ[\x-\cc] \le \rho^2\}$ is an empty ellipsoid,
then $H_{\cE}$ contains all of $\cD(\Z^n)$. Thus, $\cD(\Z^n) \subset \underset{\{\cE\:
\textrm{is\: empty}\}}{\bigcap} H_{\cE}$. The right hand side is the intersection of
infinitely many halfspaces, which can be replaced by the intersection of only those
halfspaces whose boundaries are completely determined by the elements of $\cD(\Z^n)$
that lie on them. Throughout the paper we use $\conv S$ to denote the convex hull of a
set  $S \subset \R^n$ and $\aff S$ to denote the minimal affine subspace containing
$S$.

 We define the \emph{Erdahl polyhedron} $\E(n)$ as the
intersection of closed halfspaces $H_{\cE}$ such that $\cE$ is empty and
$\rank(\partial H_{\cE} \cap \cD(\Z^n)) = \dim \cF_0(n).$

\begin{theorem} The set $\cD(\Z^n)$ coincides with $\cD(\R^n) \cap \partial {\E}(n)$.
Furthermore,

\[ \conv \cD(\Z^n)= \E(n).\]
\end{theorem}
\begin{proof} Each $\z \in \Z^n$ belongs to the \emph{boundary} of an empty degenerate
ellipsoid \begin{equation*} \cE=\{\x \in \R^n\;\; \vline \;\; \z \cdot \e_1 \le \x
\cdot \z \le \z\cdot \e_1+1 \}. \end{equation*} It is easy to check that $\partial \cE$
is the only quadratic surface passing through $\Z^n \cap
\partial \cE$ (or see \cite{erdahlcone} for a proof). Thus, $\rank
\partial H_{\cE} \cap \cD(\Z^n)=\dim \cF_0(n)$ and $\cD(\Z^n)
\subset \partial {\E}(n)$.

Let  $X \in \cD(\R^n) \cap \partial {\E}(n)$. Then there is $\x \in \R^n$ with
$X=\cD(\x)$.  If $\x \notin \Z^n$ (for $\x \in \Z^n$ see above), then $\x=\uu+\x'$,
where $\uu \in \Z^n$ and $\x' \in [0,1]^n$. Without loss of generality assume that
$x_1' \notin \Z$. Then $\x \in \inter \cE$, where $\cE=\{\x\; \vline\;u_1 \le x_1 \le
u_1+1\}$, which means $X \notin H_{\cE}$, contradicting our choice of $X$. Thus, the
only points of $\R^n$ that are mapped by $\cD$ on $\partial \E(n)$ are elements of
$\Z^n$ and the first claim of the theorem is proven.

Suppose $X \in \E(n)$ and let us show $X \in \conv \cD(\Z^n)$. It is enough to prove
this implication for $X \in \partial \E(n)$. If $X \in \partial \E(n)$, then $X$ lies
on some $H_{\cE}$, where $\cE$ is empty and completely determined by elements of $\Z^n$
that lie on its boundary.  If $X \notin \conv \cD(\cE \cap \Z^n)$, then there is a
facet $C$ of  $\conv \cD(\cE \cap \Z^n)$ such that $X$ and the relative interior of
$\conv \cD(\cE \cap \Z^n)$ lie in the different halfspaces of $H_{\cE}$ with respect to
$\aff C$. Let $f_{\cE}(x)\ge 0$ be an affine inequality defining $H_{\cE}$ and let
$g_C(x)=0$ be an affine equation of the hyperplane passing through the point $I$ (where
$I:\x \mapsto \x^T\x$) and $\aff C$ such that $g_C(X)<0$. The equation of any
hyperplane in $\cF_0(n)$ passing through $\aff C$ can be written as $f_{\cE}(x)+\theta
g_C(x)=0$ for some $\rho \in \R$. Let us define
\[\theta_{m}=\sup \{\theta \in \R \;\; \vline\;\; \forall \z \in \Z^n\,f_{\cE}(\cD(\z))+\theta g_C(\cD(\z))\ge 0 \}\]
We claim that $\theta_{m}>0$ and there exists $\uu \in \Z^n \setminus E$ such that
\[f_{\cE}(\cD(\uu))+\theta_{m} g_C(\cD(\uu))= 0.\] The proof of these claims (which we
omit due to the space limitations) is based on standard techniques of geometry of
numbers and follows the line of argument used by Voronoi in his first memoir
\cite[Pages~177--179]{VC}. $\Box$

\end{proof}

 If an ellipsoid $\cE$ is empty and $\partial H_{\cE}$ contains $\dim \cF_0(n)+1$ affinely independent
points, then $\cE$ is uniquely determined by the points of $\Z^n$ that lie on its boundary:  in this case $\cE$ is called a \emph{perfect ellipsoid}
for lattice $\Z^n$. Perfect ellipsoids were introduced by Erdahl \cite{E75,erdahlcone}. Thus,
\[\E(n)=\conv \cD(\Z^n)=
\underset{\{E\: \textrm{is\: perfect}\}}{\bigcap} H_{\cE}\: .\]

 Note that $\E(n)$  is not a polyhedron in the sense of linear programming, where
 the number of constrains is always assumed to be finite. We will refer to the faces of
 $\E(n)$ of
dimension $\dim \cF_0(n)-1$ as facets and the facets of dimension
$\dim \cF_0(n)-1$
  as faces. The facets of $\E(n)$ correspond to perfect Delaunay polyhedra.
The bounded facets of $\E(n)$ correspond to  perfect Delaunay polytopes. Two perfect Delaunay polytopes are adjacent if the corresponding  facets of
$\E(n)$ share a bounded ridge. Faces of $\E(n)$ correspond to Delaunay polyhedra -- bounded faces to bounded Delaunay polyhedra (polytopes) and
unbounded faces to unbounded polyhedra. There is a great deal of analogy between $\E(n)$  and Voronoi's polyhedron $\Pi(n)$, introduced by Venkov
\cite{venkov} (see also \cite{RB79}). Recall that $\Pi(n)$ is defined as the convex hull of
 $\{\cV(\p)\;|\;\p \in \Z^n\; \textrm{and}\; g.c.d.(p_1,\ldots,p_n)=1\}$, where $\cV: \R^n \rightarrow \Sym(n,\R)$ is the
Voronoi map. The facets of $\Pi(n)$ are defined by closed halfspaces corresponding to
\emph{perfect forms}, which were studied by Voronoi \cite{Vor1} (part I). The bounded
facets of $\Pi (n)$ correspond to positive definite perfect forms.

\subsection{Geometry of $\E(n)$}

Denote by $Aff_n(\Z)$ the group of affine automorphisms of $\Z^n$,
 i.e. the group of transformations of the form $A(\z)= L\z+\bt$, where $L \in
GL_n(\Z)$ and $\bt \in \Z^n$. The action of $Aff_n(\Z)$ on $\R^n$
can be naturally lifted to $\cF_0(n)$ by
\[F \mapsto \{\x \mapsto F(A^{-1}(\x))\}.\]

The group $Aff_n(\Z)$ acts on $\E(n)$ in a way somewhat similar to that of $GL_n(\Z)$
acting on $\Pi(n)$. Subsets $V$ and $V'$ of $\R^n$ are called arithmetically equivalent
if there exists $A \in Aff_n(\Z)$ such that $A(V)=V'$. Obviously, arithmetic
equivalence preserves properties of ellipsoids such as the Delaunay property,
emptiness, and perfection. Since there are only finitely many arithmetically distinct
Delaunay polytopes in each dimension (e.g. \cite{DL97}), the boundary of $\E(n)$ has
finitely many distinct arithmetic types of faces. In fact, the definition of perfect
ellipsoid implies that arithmetically equivalent perfect ellipsoids are isometric.

There are beautiful connections between the polytope $\E(n)$ and
Delaunay tilings of $\Z^n$. The projection $\Lin:\cF_0(n)
\rightarrow \R^n$ maps the vertices of $\partial \E(n)$ onto the
points of $\Z^n$. The projection of each face of $\E(n)$ is a
Delaunay polyhedron in $Del(\Z^n,\cQ)$ for some positive quadratic
form $\cQ$. In particular, the projections of facets of $\E(n)$ are
perfect Delaunay polyhedra.
\section{Algorithm}
In this paper we present a practical algorithm that finds all
perfect Delaunay polytopes that belong to the adjacency component of
a known $n$-dimensional perfect Delaunay polytope.
Our algorithm is best explained geometrically in terms of the geometry of $\cF_0(n)$,
although it is easier to implement it in terms of $\cF(n)$ by representing the closed
halfspace $H_{\cE}$ corresponding to an empty ellipsoid
$\cE=\{\x\in\R^n\;\vline\;F(\x)\le 0\}$ by the ray $\R_{+}F$ in $\cF(n)$. In this dual
interpretation we  consider the convex hull $\mathcal{C}(n)$ in $\cF(n)$ of all rays
corresponding to all empty ellipsoids of $\E(n)$. Each extreme ray of the cone
$\mathcal{C}(n)$ is of the form $\R P$, where $P$ is a perfect quadratic function.
Thus, the adjacency between the facets of $\E(n)$ corresponds to the adjacency between
the extreme rays of  the cone $\mathcal{C}(n)$, i.e., facets of $\E(n)$ determined by
perfect functions $F$ and $F'$ are adjacent if an only if the rays $\R F$ and $\R F'$
share a common \emph{$2$-face} of the cone $\mathcal C(n)$. The convenience of this
representation for computing is much due to its homogeneity, that is, affine
transformations of $\R^n$ induce linear transformations on $\cF(n)$.

We record our knowledge of the adjacency component under investigation in the adjacency graph $G(V,E)$, where $V$ is  the set of arithmetic types of
perfect Delaunay polytopes and $\cE$ is  the set of arithmetic types of pairs $(P,P')$, where $P$ and $P'$ are perfect Delaunay polytopes and $\qrank
(\vertex P \cap \vertex P')=2$. Equivalently, one can think of  $V$ as of a set of inequivalent facets of $\E(n)$ and $\cE$ as the a set of
inequivalent ridges of $\E(n)$. In the following subsection we describe the basic step of the algorithm.
\subsection{Step of the Algorithm}
 Let $P \in Del(\Z^n,\cQ)$ be a perfect Delaunay
polytope and let $\cE=\{\x \in \R^n\;|\;F(\x) \le \alpha,\;F \in \cF_0(n)\}$ be its
empty circumscribed ellipsoid. We regard $P$ as a vertex of the adjacency graph
$G(V,E)$. At each moment the algorithm is looking at a particular vertex of this graph.
First we find a subset $S$ of $\vertex P$ with $\qrank S=2$. Let $H_{\cE}=\{X \in
\cF_0(n)\;|\;F\cdot X\ge \alpha\}$ and let $G \cdot X = \beta$ be the equation of the
hyperplane in $\cF_0(n)$ passing through the point $I$ (where $I(\x)=\x^T\x$) and $\aff
\cD(S)$ such that $G \cdot \cD(\vv) \le \beta$ for all $\vv \in \vertex P$. The
equation of any hyperplane passing through $\aff \cD(S)$ can be written as $F\cdot
X+\rho G\cdot X =\alpha+\rho \beta$ for some $\rho \in \R$. Let
\[\rho_{m}=\sup \{\rho \in \R \;\; \vline\;\;\forall \z \in \Z^n \;
F\cdot \cD(\z) +\rho G\cdot\cD(\z)\ge \alpha+\rho \beta\:\}.\] In situations like this
it is often said that the hyperplanes $\{\cH(\rho)\}$, where
\[\cH(\rho)=\{X\;\vline\;F\cdot X+\rho G \cdot X=\alpha+\rho \beta\},\] \emph{hinge} on
the ridge $\aff \cD(S) \cap \E(n)$ of the surface $\partial \E(n)$ and that $\rho$ is
the \emph{hinge parameter.}

 It can be shown (see Theorem 3) there exists  $\uu \in \Z^n
\setminus P$ such that
\[F\cdot \cD(\uu) +\rho_{m} G \cdot \cD(\uu) \ge \alpha+\rho_{m} \beta.\]
The search for $\uu \in \Z^n \setminus P$ such that $F\cdot \cD(\uu) +\rho_{m} G \cdot
\cD(\uu) \ge \alpha+\rho_{m} \beta$ can be interpreted as continuous rotation of the
hyperplane $\cH(\rho)$ from the initial position at $\rho=0$ to the final position at
$\rho=\rho_{m}$ (see Figure 1). For small values of $\rho$  the hyperplane $\cH(\rho)$
intersects with $\E(n)$ only over the ridge $\aff \cD(S)\cap \E(n)$. When $\rho$
reaches the value of $\rho_{m}$ the rotational motion of the hyperplane is stopped by
the point $\cD(\uu)$. $S$ and $\uu$ define a perfect quadratic function and
corresponding perfect Delaunay polyhedron. If the new polyhedron is a \emph{polytope,}
we check whether it is arithmetically equivalent to any of the already discovered
polytopes. If it is a polytope distinct from the previously discovered ones, we add it
to the list of perfect $n$-polytopes and update the adjacency graph.
\begin{figure*}[t]
\begin{center}
\resizebox{!}{300pt}{\includegraphics[clip=false,keepaspectratio=true]{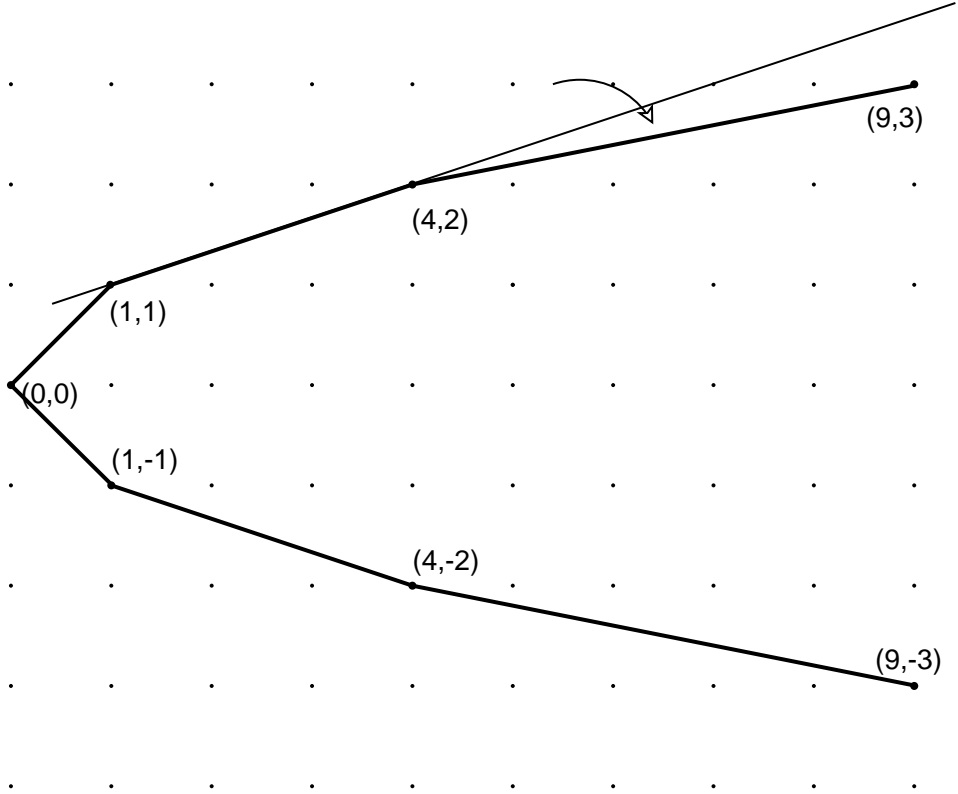}} \caption{Going from facet $\conv\{(1,1),(4,2)\}$ to facet
$\conv\{(4,2),(9,3)\}$. }
\end{center}
\end{figure*}
This procedure is similar to the one used by Voronoi in the determination of perfect
forms in small dimensions. He referred to this procedure as the method of continuous
variation of parameters. The geometric interpretation of Voronoi's method as that of
hinging  hyperplanes was given by Venkov \cite{venkov}. Later this method was
rediscovered in the context of polytopes by \cite{ck-acp-70} and dubbed as ``gift
wrapping method".

The procedures described above paragraph are repeated for each
arithmetic class of subsets of $\vertex P$ of quadratic rank $2$.
When all such subsets are exhausted, we move to another vertex of
the adjacency graph $G(V,E)$

\subsubsection{Finding $\rho_{m}$ and $\uu$}
 Let $S \subset \vertex P$ and let $\qrank S=2$.  Using some heuristic we pick
 some $\z \notin \vertex P$ with $G \cdot \cD(\z) > \beta$
 and construct an ellipsoid $\cE$ through $S$
 and $\z$. We find its center $\cc$ and then look for the closest
 lattice point to $\cc$ in the metric defined by $\cE$. This
test  can be done efficiently for $n \le 9$ using the program \texttt{Lattice-CVP} by Dutour (see \cite{CVP}). If the closest lattice points happens
to be
 at the same distance from $\cc$ as $\z$ and $S$, then we declare
 $\Z^n \cap \cE$ the vertex set of a perfect Delaunay polyhedron. If
 the interior of $\cE$ contains a lattice point $\z'$, then we abandon $\z$ and repeat
 the computation for $S$ and $\z'$, etc.
 \subsection{Using Symmetries in computation}
Our algorithm would be impractical if we failed to use symmetries in an efficient way. Two isomorphism problems had to be addressed. The first is the
problem of checking whether two perfect Delaunay polytopes are arithmetically equivalent. The second problem is finding all arithmetically
inequivalent subsets $S$ of $\vertex P$ with $\qrank S=2$. Algorithms for these problems have been implemented in GAP (some of them are available in
\cite{DuPol}) and rely on the use of the program \texttt{nauty} \cite{McKay}. See also \cite{DutourAdj}.

\subsection{Results}
The graph $G(V,E)$ constructed by the algorithm encodes the adjacency pattern for
\emph{bounded} facets of $\E(n)$. More generally, denote by
$\overline{G}(\overline{V},\overline{E})$ the graph whose vertices are the arithmetic
types of facets of $\E(n)$ and whose edges are arithmetic types of pairs of facets
sharing a common ridge. As of now this graph is completely known only for $n\le 6$. For
$n=6$ $\overline{G}(\overline{V},\overline{E})$ has two vertices, which correspond to
the Gosset $6$-polytope $2_{21}$ and the unit slab $U$. The quadratic form for the
Gosset $6$-polytope is  $E_6$ and that for the unit slab is a rank one form (see Figure
2).
\begin{figure}[h]
\begin{center}
\resizebox{!}{50pt}{\includegraphics[clip=false,keepaspectratio=true]{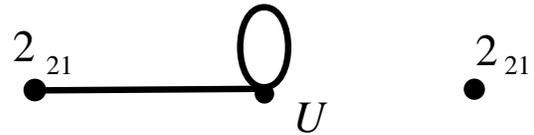}}
\caption{$n=6$. left: $\overline{G}(\overline{V},\overline{E})$,
right: $G(V,E)$.}
\end{center}
\end{figure}

 For $n=7$ the discovered adjacency component of $\overline{G}(\overline{V},\overline{E})$
 has 4 vertices and that of $G(V,E)$ has two vertices (see Figure 3). The latter
 two vertices correspond to the Gosset 7-polytope $3_{21}$ and a polytope
 with 35 vertices discovered earlier by Erdahl and Rybnikov (see \cite{er1}).

\begin{figure*}[t]
\begin{center}
\resizebox{!}{100pt}{\includegraphics[clip=false,keepaspectratio=true]{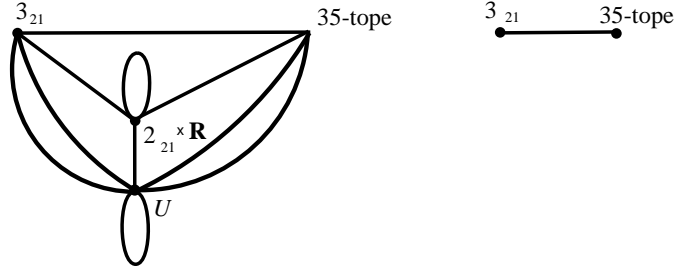}}
\caption{$n=7$.  left: conjectured
$\overline{G}(\overline{V},\overline{E})$, right: conjectured
$G(V,E)$.}
\end{center}
\end{figure*}
 For $n=8$ we have determined an adjacency component of the restricted graph $G(V,E)$.
 Below is the adjacency list of  the conjectured $G(V,E)$ in GAP format. Note that the graph has loops and multiple edges.

 \par \noindent
{\small {\bf 1:} [1, 2]
 \par \noindent {\bf 2:} [2, 16], [2, 27], [2, 8], [2, 10], [2, 22], [2, 4], [2, 5], [2, 13], [2,~7], [2, 6],
[2, 3],  [2, 14], [2, 12], [2, 19], [2, 9], [2, 18], [2, 8], [2, 11], [2, 15], [2, 6],
[2, 11], [2, 2], [2, 17], [2, 1]
 \par \noindent {\bf
3:} [3, 20], [3, 13], [3, 12], [3, 3], [3, 11], [3, 4], [3, 14], [3, 5], [3,~15], [3,
10], [3, 6], [3, 2]
 \par \noindent {\bf
4:} [4, 5], [4, 6], [4, 10], [4, 22],  [4, 3], [4, 2], [4, 8], [4, 14], [4,~20], [4,
19]
 \par \noindent {\bf
5:} [5, 9], [5, 21], [5, 6], [5, 10], [5, 5], [5, 9], [5, 6], [5, 22], [5, 8], [5, 20],
[5, 4], [5, 3], [5, 2]
 \par \noindent {\bf
6:} [6, 22], [6, 10], [6, 5], [6, 4], [6, 10], [6, 3], [6, 2], [6, 9], [6, 6], [6, 5],
[6, 21], [6, 8], [6, 12], [6, 15], [6, 24], [6, 6], [6, 7], [6, 11], [6, 2], [6, 15]
 \par \noindent {\bf
7:} [7, 7], [7, 12], [7, 21], [7, 22],  [7, 9], [7, 24], [7, 19], [7, 8], [7,~6], [7,
10], [7, 2]
 \par \noindent {\bf
8:} [8, 22], [8, 2], [8, 27], [8, 8], [8, 16], [8, 8], [8, 10], [8, 20], [8,~9], [8,
6], [8, 2], [8, 5], [8, 8], [8, 4], [8, 13], [8, 12], [8, 7]
 \par \noindent {\bf
9:} [9, 8], [9, 6], [9, 5], [9, 22], [9, 5], [9, 7], [9, 2], [9, 10], [9, 19], [9, 21],
[9, 23], [9, 20], [9, 22]
 \par \noindent {\bf
10:} [10, 6], [10, 5], [10, 15], [10, 24], [10, 10], [10, 22], [10,~9], [10, 12], [10,
21], [10, 7], [10, 10], [10, 10], [10, 3], [10, 2], [10,~20], [10, 8], [10, 4], [10,
26], [10, 6], [10, 12], [10, 13], [10,~16], [10, 16]
 \par \noindent {\bf 11:} [11, 6], [11, 3], [11, 2], [11, 2]
 \par \noindent {\bf 12:} [12, 10], [12, 20], [12, 3], [12, 8], [12, 10], [12, 2],  [12, 7], [12, 6], [12, 21]
 \par \noindent {\bf 13:} [13, 10], [13, 20], [13, 3], [13, 8], [13, 2]
 \par \noindent {\bf 14:} [14, 4], [14, 18], [14, 3], [14, 2]
 \par \noindent {\bf 15:} [15, 2], [15, 6], [15, 6], [15, 3], [15, 10], [15, 21]
 \par \noindent {\bf 16:} [16, 8], [16, 2], [16, 27], [16, 20], [16, 10], [16, 10],
 \par \noindent {\bf 17:} [17, 2]
 \par \noindent {\bf 18:} [18, 19], [18, 2], [18, 14]
 \par \noindent {\bf 19:} [19, 9], [19, 7], [19, 2], [19, 4], [19, 20], [19, 18], [19, 25]
 \par \noindent {\bf 20:} [20, 16],  [20, 22], [20, 9], [20, 5], [20, 19], [20, 4], [20, 20], [20,~12], [20, 3], [20, 13], [20, 10], [20, 8]
 \par \noindent {\bf 21:} [21, 7],  [21, 12], [21, 24], [21, 21], [21, 22], [21, 15], [21, 5], [21,~9], [21, 23], [21, 26], [21, 6], [21, 10]
 \par \noindent {\bf 22:} [22, 22], [22, 2], [22, 8], [22, 6], [22, 27], [22, 9], [22, 7], [22,~4], [22, 5], [22, 10], [22, 9], [22, 22], [22, 21],  [22, 20],
 \par \noindent {\bf 23:} [23, 9], [23, 21], [23, 25]
 \par \noindent {\bf 24:} [24, 6], [24, 21], [24, 7], [24, 10]
 \par \noindent {\bf 25:} [25, 23], [25, 19]
 \par \noindent {\bf 26:} [26, 10], [26, 21]
 \par \noindent {\bf 27:} [27, 8], [27, 22], [27, 16], [27, 2]}

 The numbers of vertices correspond to the numbers of
 polytopes in \cite{DER} where a complete analysis of existing data
 on perfect Delaunay polyhedra for $n\le 8$ is given.

 In
practice the algorithm often encounters perfect ellipsoids equivalent to the unit slab,
i.e., to the set $0\le  x_1 \le 1$. For $n\in\{6,7,8\}$ we know that any perfect
polyhedron adjacent to the unit slab is either a unit slab or the product of Gosset's
$6$-polytope $2_{21}$ and $\R^{n-6}$. However, for $n>6$ other unbounded perfect
polyhedra appear, such as e.g. the product of $2_{21}$ and $\R^2$ for $n=8$. We are not
able, at the present, to comprehensively  handle all subsets $S$  with $\qrank S=2$ for
such polyhedra. That is why  we cannot formally claim that our results for $n=7,8$ are
complete.

\section{Discussion}

The new method  has many advantages over the one of
\cite{DutourAdj}:
\begin{enumerate}
\item Unlike the previous methods (see e.g. \cite{DutourAdj}), the new method uses
the full symmetry group of $P$.  In particular, the  use of the full symmetry group of
$P$ allows us to use the Recursive Adjacency Decomposition Method of
\cite{bremnerpaper} to terminate computations. The termination problem is an important
one. Previous methods did not have a satisfactory solutions to the termination problem.
\item Previous methods had to select an affine basis for each perfect Delaunay
polytope. We do not know if it is possible to find an affine basis
for every perfect Delaunay polytope. The method of this paper does
not require this assumption.
\item Our method is no longer reduced to {\em basic Delaunay polytopes}.
 We know that there exist non-basic Delaunay polytopes (see \cite{grishukhin}) and
 we cannot exclude the possibility that  there exist non-basic perfect Delaunay polytopes.
 \item The new method has found all presently known $8$-dimensional perfect Delaunay polytopes. The method of \cite{DutourAdj} run in dimension $8$
 does not find some of these polytopes -- they were found as sections of higher-dimensional perfect Delaunay polytopes obtained earlier by the old
 method; however, these sections were found by a heuristic approach without any guarantee of completeness. On the other hand, if Conjecture 1 is true, then the new algorithm
 has provably found all perfect Delaunay $n$-polytopes for $n\le 8$.   We expect that running the new algorithm for $n=9,10$ will uncover previously unknown
 perfect polytopes in these dimensions.
\end{enumerate}
Our method cannot deal at present with  unbounded perfect Delaunay polyhedra. When a perfect polyhedron $P$ is unbounded, but not equivalent to the
unit slab, it is difficult to find all equivalence classes of its vertex subsets corresponding to the ridges of $\E(n)$. For this reason we cannot
guarantee that our algorithm has found all perfect ellipsoids in dimensions 7 and 8.
\bibliographystyle{latex8}

\end{document}